\documentclass[12pt]{article}
\usepackage{amsmath}
\usepackage{float}
\usepackage{amsthm}
\usepackage{amssymb}
\usepackage[utf8x]{inputenc}
\usepackage{caption}
\usepackage{graphics}
\usepackage{enumerate} 
\usepackage{algorithmic}
\usepackage{calrsfs}
\DeclareMathAlphabet{\pazocal}{OMS}{zplm}{m}{n}

\usepackage[Algorithm,ruled]{algorithm}
\usepackage{titlesec}
\usepackage{sectsty}
\titleformat*{\section}{\LARGE\bfseries}
\titleformat*{\subsection}{\Large\bfseries}
\titleformat*{\subsubsection}{\large\bfseries}
\sectionfont{\Huge}
\subsectionfont{\large}
\subsubsectionfont{\normalsize}
\usepackage{multicol}
\usepackage{lmodern}
\usepackage{marvosym} 
\usepackage{lipsum}
\usepackage{mwe}
\usepackage{caption}
\usepackage{subcaption} 
\newtheoremstyle{case}{}{}{}{}{}{:}{ }{}
\theoremstyle{case}

\newcommand{\be}{\begin{equation}}
\newcommand{\ee}{\end{equation}}
\newcommand{\ben}{\begin{eqnarray*}}
\newcommand{\een}{\end{eqnarray*}}
\newtheorem{examp}{\sc example}
\newtheorem{remk}{\sc remark}
\newtheorem{corol}{\sc corollary}
\newtheorem{lemma}{\sc lemma}
\newtheorem{theorem}{\sc theorem}
\newtheorem{defn}{\sc definition}
\newcommand{\bt}{\begin{theorem}}
\newcommand{\et}{\end{theorem}}
\newcommand{\bl}{\begin{lemma}}
\newcommand{\el}{\end{lemma}}
\newcommand{\bed}{\begin{defn}}
\newcommand{\eed}{\end{defn}}
\newcommand{\brem}{\begin{remk}}
\newcommand{\erem}{\end{remk}}
\newcommand{\bex}{\begin{examp}}
\newcommand{\eex}{\end{examp}}
\newcommand{\bcl}{\begin{corol}}
\newcommand{\ecl}{\end{corol}}

\newcommand{\NI}{\noindent}

\theoremstyle{definition}
\theoremstyle{remark}

\numberwithin{equation}{section}
\numberwithin{theorem}{section}
\numberwithin{lemma}{section}

\begin{document}

\title{\large\bf\sc More on semipositive tensor and tensor complementarity problem}

\author{%A. Dutta$^{a,1}$, 
R. Deb$^{a,1}$ and A. K. Das$^{b,2}$\\
\emph{\small $^{a}$Jadavpur University, Kolkata , 700 032, India.}\\	
\emph{\small $^{b}$Indian Statistical Institute, 203 B. T.
	Road, Kolkata, 700 108, India.}\\
%\emph{\small $^{1}$Email: rwitamjanaju@gmail.com}\\
%\emph{\small $^{1}$Email: aritradutta001@gmail.com}\\
\emph{\small $^{1}$Email: rony.knc.ju@gmail.com}\\
\emph{\small $^{2}$Email: akdas@isical.ac.in}\\
}

\date{}

\maketitle

\begin{abstract}
\noindent In recent years several classes of structured matrices are extended to classes of tensors in the context of tensor complementarity problem. The tensor complementarity problem is a class of nonlinear complementarity problem where the involved functions are special polynomials defined by a tensor. Semipositive and strictly semipositive tensors play an important role in the study of the tensor complementarity problem. The article considers some important properties of semipositive tensor. We establish invariance property of semipositive tensor. We prove necessary and sufficient conditions for a tensor to be semipositive tensor. A relation between even order row diagonal semipositive tensor and its majorization matrix is proposed.\\

\\

\end{abstract}
\footnotetext[1]{Corresponding author}

\section{Introduction}

A tensor is a multidimensional array which is a natural extension of matrices. A real tensor of order $r$ and dimension $n$, $\mathcal{M}= (m_{i_1 ... i_r}) $ is a multidimensional array of entries $m_{i_1 ... i_r} \in \mathbb{R}$ where $i_j \in I_n$ with $j\in I_r.$ Here $I_n$ denotes the set $I_n=\{1,2,...,n\}.$ $\mathbb{T}_{r,n}$ denotes the set of real tensors of order $m$ and dimension $n.$ %Any $\mathcal{A}= (a_{i_1 i_2 ... i_m}) \in T_{m,n} $ is called a symmetric tensor, if the entries $a_{i_1 i_2 ... i_m}$ are invariant under any permutation of their indices. $S_{m,n}$ denotes the collection of all symmetric tensors of order $m$ and dimension $n$  where $m$ and $n$ are two given positive integers with $m,n\geq 2$.
Tensors have many application in science and engineering. The common applications are found in electromagnetism, continuum mechanics, quantum mechanics and quantum computing, spectral hypergraph theory, diffusion tensor imaging, image authenticity verification problem, optimization theory and in many other areas. In optimization theory the tensor complementarity problem (TCP) was proposed by Song and Qi \cite{song2017properties}, which is a natural extension of the linear complementarity problem and a subclass of the nonlinear complementarity problem where the involved functions are defined by a tensor. 
\NI Given a square matrix $A$ of order $n$ with real entries and an $n$ dimensional vector $q$, the linear complementarity problem \cite{cottle2009linear} is to find $n$ dimensional vectors $w$ and $z$ satisfying
\begin{equation}\label{linear comp equation}
 z\geq 0,\;\; w = Az + q \geq 0,\;\; z^T w = 0.
\end{equation}
The problem is denoted by LCP$(q,A)$ and the solution set of LCP$(q,A)$ is denoted by SOL$(q,A).$
Several matrix classes arising in linear complementarity problem are important due to computational point of view. For details see \cite{neogy2006some}, \cite{neogy2013weak}, \cite{neogy2005almost}, \cite{neogy2011singular}, \cite{jana2019hidden}, \cite{jana2021more}, \cite{neogy2009modeling}, \cite{das2017finiteness}. For details of game theory see \cite{mondal2016discounted}, \cite{neogy2008mathematical}, \cite{neogy2008mixture}, \cite{neogy2005linear}, \cite{neogy2016optimization}, \cite{das2016generalized} and for details of QMOP see \cite{mohan2004note}. Even several matrix classes arise during the study of Lemke's algorithm as well as principal pivot transform. For details see \cite{mohan2001classes}, \cite{mohan2001more} \cite{neogy2005principal}, \cite{das2016properties}, \cite{neogy2012generalized}, \cite{jana2019hidden}, \cite{jana2021more}, \cite{jana2018processability}.

%For a tensor $\mathcal{M}\in \mathbb{T}_{r,n} $ and $u\in \mathbb{R}^n,~ \mathcal{M}u^{r-1}\in \mathbb{R}^n $ is a vector defined by $(\mathcal{M}u^{r-1})_i = \sum_{i_2, ...i_r =1}^{n} m_{i i_2  ...i_r} u_{i_2} \cdot\cdot \cdot u_{i_r} , \;\forall \; i \in I_n,$ and $\mathcal{M}u^r\in \mathbb{R} $ is a scalar defined by $ u^T \mathcal{M}u^{r-1} = \mathcal{M}u^r = \sum_{i_1, ...i_r =1}^{n} m_{i_1  ...i_r} u_{i_1} \cdot\cdot \cdot u_{i_r}.$
\NI Given a tensor $\mathcal{M}\in\mathbb{T}_{r,n}$ and an $n$ dimensional real vector $q$, the tensor complementarity problem is to find $n$ dimensional vectors $\omega$ and $u$ satisfying
	\begin{equation}\label{tensor comp equation}
	u\geq 0,\;\; \omega =\mathcal{M}u^{r-1}+q \geq 0,\;\;  u^T \omega= 0.
	\end{equation}
The problem is denoted by TCP$(q,\mathcal{M}).$ The solution set of TCP$(q,\mathcal{M})$ is denoted by SOL$(q,\mathcal{M}).$
Tensor complementarity problem arises in optimization theory, game theory and in other areas. Many practical problems can be modeled as forms of equivalent tensor complementarity problem like a class of multi-person noncooperative game \cite{huang2017formulating}, hypergraph clustering problem and traffic equilibrium problem \cite{huang2019tensor}. Tensor classes play an important role to study the theory of TCP. In recent years many structured tensors are developed and studied in context of TCP. For details see \cite{qi2005eigenvalues, song2014properties, song2015properties, song2016properties, palpandi2021tensor, luo2017sparsest}. In linear complementarity problem the class of semimonotone matrices play an important role. Several matrix theoretic properties and properties related to solution of linear complementarity problem involving semimonotone matrices were studied in literature. For details see \cite{karamardian1972complementarity, pang1979onq, eaves1971linear, danao1997note, tsatsomeros2019semimonotone}.
Song and Qi \cite{song2014properties} extended the class of semimonotone matrices to semipositive tensors in context of tensor complementarity problem. Song and Qi \cite{song2014properties} proved that a strictly semipositive tensor is a $Q$ tensor. Song and Qi \cite{song2016tensor} showed that for a (strictly) semipositive tensor $\mathcal{A}$ the TCP$(q,\mathcal{M})$ has a unique solution for every $q>0\; (q \geq 0).$ Many other tensor theoretic properties as well as properties of solution of TCP with semipositive tensor were discussed.
For details see \cite{song2014properties, song2016tensor, song2016properties, guo2019properties, zheng2018class}. 

\NI In this paper, we study some tensor theoretic properties of semipositive tensor as well as strictly semipositive tensor. We discuss some necessary and sufficient condition for (strictly) semipositive tensor. We prove an equivalence in connection with the majorization matrix of an even order row diagonal (strictly) semipositive tensor.

%We establish some tensor theoretic properties of (strictly) semipositive tensor. We study the invariance  We propose an equivalence between a even order row diagonal semipositive tensor and a semimonotone matrix. 

The paper is organised as follows. Section 2 contains some basic notations and results. In Section 3, we investigate some tensor theoretic properties of semipositive tensor. We establish invariance property of semipositive tensor. We propose necessary and sufficient condition for semipositive tensor. We establish connection between an even order row diagonal semipositive tensor and its majorization matrix.

\section{Preliminaries}

Now we introduce some basic notations used in this paper. We consider vectors, matrices and tensors with real entries. For any positive integer $n,$ let $I_n$ denote the set $\{ 1, 2,...,n \}$. Let $\mathbb{R}^n$ denote the $n$-dimensional Euclidean space and $\mathbb{R}^n_+ =\{ u\in \mathbb{R}^n : u\geq 0 \}$, $\mathbb{R}^n_{+ +} =\{ u\in \mathbb{R}^n : u> 0 \}$. Any vector $u\in \mathbb{R}^n$ is a column vector and $u^T$ denotes the row transpose of $u.$ A diagonal matrix $D=[d_{ij}]_{n \times n}=diag(d_1, \; d_2,\; ..., \; d_n)$ is defined as $d_{ij}=\left \{ \begin{array}{ll}
	  d_i  &;\; \forall \; i=j, \\
	  0  &; \; \forall \; i \neq j.
	   \end{array}  \right.$
	   
\begin{defn}\cite{pang1979onq, cottle1968problem}
A matrix $M\in \mathbb{R}^{n\times n}$ is said to be semimonotone, if for every $0\neq z \geq 0,\; \exists$ an index $k\in I_n$ such that $z_k >0$ and $(M z)_k \geq 0.$
\end{defn}
\begin{defn}\cite{pang1979onq, cottle1968problem}
A matrix $M\in \mathbb{R}^{n\times n}$ is said to be strictly semimonotone, if for every $0\neq z \geq 0,\; \exists$ an index $k\in I_n$ such that $z_k >0$ and $(M z)_k > 0.$
\end{defn}
\NI Let $r$th order $n$ dimensional real tensor $\mathcal{M}= (m_{i_1 ... i_r}) $ be a multidimensional array of entries $m_{i_1 ... i_r} \in \mathbb{R}$ where $i_j \in I_n$ with $j\in I_r$. $\mathbb{T}_{r,n}$ denotes the set of real tensors of order $r$ and dimension $n.$ Any $\mathcal{M}= (m_{i_1 ... i_r}) \in \mathbb{T}_{r,n} $ is called a symmetric tensor, if the entries $m_{i_1 ... i_r}$ are invariant under any permutation of their indices. $\mathbb{S}_{r,n}$ denotes the collection of all symmetric tensors of order $r$ and dimension $n$  where $r$ and $n$ are two given positive integers with $r,n\geq 2$. An identity tensor of order $r,$ $\mathcal{I}=(\delta_{i_1 ... i_r})\in \mathbb{T}_{r,n}$ is defined as follows:
$\delta_{i_1 ... i_r}= \left\{
\begin{array}{ll}
	  1  &:\; i_1= ...= i_r \\
	  0  &:\; else
	   \end{array}
 \right. .$
Let $\mathcal{O}$ denote the zero tensor where each entry of $\mathcal{O}$ is zero. For any $u \in \mathbb{R}^n$ and $p\in \mathbb{R}$, let $u^{[p]}$ denote the vector $(u_1^{p},\; u_2^{p},\;...,\;u_n^{p})^T.$ For $\mathcal{M}\in \mathbb{T}_{r,n} $ and $u\in \mathbb{R},\; \mathcal{M}u^{r-1}\in \mathbb{R}^n $ is a vector defined by
\[ (\mathcal{M}u^{r-1})_i = \sum_{i_2, ..., i_r =1}^{n} m_{i i_2 ...i _r} u_{i_2}  \cdot\cdot \cdot u_{i_r} , \mbox{   for all } i \in I_n \]
and $\mathcal{M}u^r\in \mathbb{R} $ is a scalar defined by
\[ u^T\mathcal{M}u^{r-1}= \mathcal{M}u^m = \sum_{i_1, ..., i_r =1}^{n} m_{i_1 ...i_r} u_{i_1} \cdot \cdot \cdot u_{i_r}. \]

\NI For a tensor $\mathcal{M}\in \mathbb{T}_{r,n}$ a vector $u\in \mathbb{R}^n$ is said to be a null vector of $\mathcal{M},$ if $\mathcal{M} u^{r-1} =0.$

\noindent The general product of tensors was introduced by Shao \cite{shao2013general}. Let $\mathcal{A}$ and $\mathcal{B}$ be two $n$ dimensional tensor of order $p \geq 2$ and $r \geq 1,$ respectively. The product $\mathcal{A} \mathcal{B}$ is an $n$ dimensional tensor $\mathcal{C}$ of order $((p-1)(r-1)) + 1$ with entries 
\[c_{j \beta_1 \cdots \beta_{p-1} } =\sum_{j_2, \cdots ,j_p \in I_n} a_{j j_2 \cdots j_p} b_{j_2 \beta_1} \cdots b_{j_p \beta_{p-1}},\] where $j \in I_n$, $\beta_1, \cdots, \beta_{p-1} \in I_n^{r-1}$. %This product was proved to be associative by Shao (\cite{shao2013general}).\

\begin{defn}\cite{qi2005eigenvalues, song2015properties}
Given $\mathcal{M}= (m_{i_1 ...i_r}) \in \mathbb{T}_{r,n}$ and an index set $J\subseteq I_n$ with $|J|=l,\; 1\leq l \leq n,$ a principal subtensor of $\mathcal{M}$ is denoted by $\mathcal{M}^J_l$ and is defined as
$ \mathcal{M}^J_l = ( m_{i_1 ...i_r} ), \; \forall \; i_1, i_2,...i_r \in J .$
\end{defn}

\begin{defn}\cite{song2016properties}
Given $\mathcal{M} \in \mathbb{T}_{r,n} $ and $q\in \mathbb{R}^n$, a vector $u$ is said to be (strictly) feasible solution of TCP$(q,\mathcal{M}),$ if $u \geq 0\; (>0)$ and $\mathcal{M}u^{r-1}+q \geq 0\; (>0)$.
\end{defn}

\begin{defn}\cite{song2016properties}
Given $\mathcal{M} \in \mathbb{T}_{r,n} $ and $q\in \mathbb{R}^n$, TCP$(q,\mathcal{M})$ is said to be (strictly) feasible if a (strictly) feasible vector exists.
\end{defn}

\begin{defn}\cite{song2016properties}
Given $\mathcal{M} \in \mathbb{T}_{r,n} $ and $q\in \mathbb{R}^n$, TCP$(q,\mathcal{M})$ is said to be solvable if there exists a feasible vector $u$ satisfying $u^{T}(\mathcal{M}u^{r-1}+q)=0$ and $u$ is said to be a solution of the TCP$(q,\mathcal{M})$.
\end{defn}

%\begin{defn}\cite{qi2013symmetric}
%A tensor $\mathcal{M}\in \mathbb{T}_{r,n} $ is said to be a strictly copositive tensor if $\mathcal{M}x^m >0 $ for all $u\in \mathbb{R}^n_+ \backslash \{0\}.$ 
%\end{defn}

%\begin{defn}\cite{qi2013symmetric}
%A tensor $\mathcal{M}\in \mathbb{T}_{r,n} $ is said to be a copositive tensor if $\mathcal{M}x^m \geq0$ for all $u\in \mathbb{R}^n_+ .$ 
%\end{defn}

\begin{defn}
\cite{song2015properties} A tensor $\mathcal{M}\in \mathbb{T}_{r,n} $ is said to be a $P_0(P)$-tensor, if for each $u\in \mathbb{R}^n \backslash \{0\}$, there exists an index $i\in I_n$ such that $u_i \neq 0$ and $u_i (\mathcal{M}u^{r-1})_i \geq 0\; (>0)$.
\end{defn}

\begin{defn}\cite{song2014properties, song2016properties}
A tensor $\mathcal{M}\in \mathbb{T}_{r,n} $ is said to be a $R_0$-tensor if the TCP$(0, \mathcal{M})$ has unique $0$ solution. 
\end{defn}

\begin{defn}\cite{song2014properties, song2016properties}
A tensor $\mathcal{M}\in \mathbb{T}_{r,n} $ is said to be a $R$-tensor if it is a $R_0$-tensor and the TCP$(e, \mathcal{M})$ has unique solution $0,$ for $e=(1,1,...,1)^T.$
\end{defn}

\begin{defn}\cite{huang2015q} 
A tensor $\mathcal{M}\in \mathbb{T}_{r,n} $ is said to be a $Q$-tensor if the TCP$(q,\mathcal{M})$ is solvable for all $q\in \mathbb{R}^n $.
\begin{center}
    i.e., $u\geq 0, ~~~ \omega= \mathcal{M}u^{r-1} + q\geq 0, ~~~\mbox{and}~~ u^{T}\omega=0.$
\end{center}
\end{defn}

\begin{defn}\cite{shao2016some}
The $i$th row subtensor of $\mathcal{M}\in \mathbb{T}_{r,n}$ is denoted by $R_i(\mathcal{M})$ and its entries are given as $(R_i(\mathcal{M}))_{i_2 ... i_r}=(m_{i i_2... i_r})$, where $i_j\in I_n$ and $2\leq j\leq r.$ $\mathcal{M}$ is said to be a row diagonal tensor if all its row subtensors $R_1(\mathcal{M}),..., R_n(\mathcal{M})$ are diagonal tensors.
\end{defn}

\begin{defn}\cite{shao2016some}
Given $\mathcal{M}=(m_{i_1 ... i_r})\in \mathbb{T}_{r,n},$ the majorization matrix of $\mathcal{M}$ is a matrix of order $n,$ denoted by $\Tilde{\mathcal{M}}$ and is defined as $\Tilde{\mathcal{M}}_{i j}= (m_{ij ... j}),$ $\forall\; i,j \in I_n.$
\end{defn}

\begin{lemma}\cite{shao2016some}
Let $\mathcal{M}\in \mathbb{T}_{r,n}$ and $\Tilde{\mathcal{M}}$ be the majorization matrix of $\mathcal{M}.$ Then $\mathcal{M}$ is row diagonal tensor if and only if $\mathcal{M}=\Tilde{\mathcal{M}}\mathcal{I},$ where $\mathcal{I}$ is the identity tensor of order $r$ and dimension $n.$
\end{lemma}

%\begin{theorem}\cite{bai2016global}
%For any $q\in \mathbb{R}^n$ and a $P$-tensor $\mathcal{M} \in \mathbb{T}_{r,n}$, the solution set of TCP$(q,\mathcal{M})$ is nonempty and compact.
%\end{theorem}
%\begin{lemma}\cite{song2014properties}
%Each strictly semipositive tensor is a $R$-tensor, and each $R$-tensor is a $Q$-tensor.
%\end{lemma}

%\begin{lemma}\cite{song2014properties}
%A semipositive $R_0$-tensor is a $Q$-tensor.
%\end{lemma}

%\begin{theorem}
%Let A be a semi-positive tensor. Then for each q ∈ Rn, the solution set of the TCP(A, q) is bounded.
%\end{theorem}

\begin{theorem}\cite{song2016tensor}
Let $\mathcal{M}=(m_{i_1, ..., i_r})\in \mathbb{T}_{r,n}.$ Then each principal subtensor of a (strictly) semipositive tensor is (strictly) semipositive tensor.
\end{theorem}

\begin{theorem}\cite{song2016tensor}
Let $\mathcal{M}=(m_{i_1, ..., i_r})\in \mathbb{T}_{r,n}.$ The following statements are equivalent:\\
(a) $\mathcal{M}$ is semipositive tensor.\\
(b) The TCP$(q,\mathcal{M})$ has a unique solution for every $q > 0.$\\
(c) For every index set $J \subseteq I_n$ with $|J|=r,$ the system $\mathcal{M}^J_r u_J^{r-1} <0,\;\; u_J \geq 0$ has no solution where, $u_J \in \mathbb{R}^{r}.$
\end{theorem}

\begin{theorem}\cite{song2016tensor}
Let $\mathcal{M}=(m_{i_1, ... i_r})\in \mathbb{T}_{r,n.}$ The following statements are equivalent:\\
(a) $\mathcal{M}$ is strictly semipositive tensor.\\
(b) The TCP$(q,\mathcal{M})$ has a unique solution for every $q\geq 0.$\\
(c) For every index set $J \subseteq I_n$ with $|J|=r,$ the system $ \mathcal{M}^J_r u_J^{r-1} \leq 0,\;\; u_J \geq 0,\;\; u_J \neq 0$ has no solution where, $u_J \in \mathbb{R}^{r}.$
\end{theorem}

%\begin{theorem}\cite{song2016tensor}
%Let $\mathcal{M}\in \mathbb{S}_{r,n}.$ Then $\mathcal{M}$ is (strictly) semipositive tensor if and only if it is (strictly) copositive tensor.
%\end{theorem}

\section{Main results}

At the outset, we define semipositive tensor along with an example. Subsequently, we establish the invariance property of semipositive tensor with respect to diagonal matrix as well as permutation matrix.

\begin{defn}\cite{song2014properties}
A tensor $\mathcal{M}\in \mathbb{T}_{r,n} $ is said to be (strictly) semipositive tensor if for each $u\in \mathbb{R}^n_+ \backslash \{0\}$, there exists an index $l\in I_n$ such that $u_l>0$ and $(\mathcal{M}u^{r-1})_l\; \geq 0\; (>0)$.
\end{defn}

Here we give an example of a semipositive tensor.
\begin{examp}
Let $\mathcal{M}\in \mathbb{T}_{4,3}$ such that $m_{1211}= 1,\; m_{1233}= 2,\; m_{1323}= -1,\; m_{1223}= -3,\; m_{1232}= 4,\; m_{2211}=1,\; m_{2223}= -3,\; m_{2322}= 5,\; m_{3232}= -1,\; m_{3322}=-3,\; m_{3313}= -2$ and all other entries of $\mathcal{M}$ are zero. For $u =\left(\begin{array}{c} u_1 \\ u_2\\ u_3 \end{array} \right) \in \mathbb{R}^3$ we have $\mathcal{M}u^3 = \left(\begin{array}{c} u_1^2 u_2 + u_2 u_3^2 + u_2^2 u_3\\ u_1^2 u_2 + 2u_2^2 u_3 \\ -2u_1 u_3^2 - 4u_2^2 u_3 \end{array} \right).$ 
% Then $u_1(\mathcal{M}u^3)_1 =-3 <0$ and $u_2(\mathcal{M}u^3)_2 =-1<0.$
Then $\mathcal{M}$ is a semipositive tensor.
\end{examp}

\begin{theorem}
Let $\mathcal{M} \in \mathbb{T}_{r,n}$ and $D\in \mathbb{R}^{n\times n}$ be a positive diagonal matrix. Then $\mathcal{M}$ is (strictly) semipositive tensor if and only if $D \mathcal{M}$ is (strictly) semipositive tensor.
\end{theorem}
\begin{proof}
Let $D=diag(d_1,...,d_n)$ be a positive diagonal matrix. Then the diagonal entries $d_i >0,\;\forall \; i\in I_n.$ Suppose $\mathcal{M} \in \mathbb{T}_{r,n}$ is a (strictly) semipositive tensor. Then for $0 \neq u\geq 0\;\exists $ an index $l$ such that $u_l>0$ and $(\mathcal{M} u^{r-1})_l \geq 0\; (>0).$ We show that $D \mathcal{M}$ is (strictly) semipositive tensor. Since $d_l> 0$ we obtain $(D \mathcal{M}u^{r-1})_l =d_l(\mathcal{M}u^{r-1})_l \geq 0\; (>0).$ Thus $D \mathcal{M}$ is (strictly) semipositive tensor.

Conversely, suppose $D \mathcal{M}$ is (strictly) semipositive tensor. Now $D$ is positive diagonal matrix so $D^{-1}$ exists and $D^{-1}$ is again a positive diagonal matrix. Therefore $D^{-1} (D\mathcal{M}) =\mathcal{M}$ is a (strictly) semipostive tensor.
\end{proof}

\begin{corol}
Let $\mathcal{M} \in \mathbb{T}_{r,n}$ be a semipositive tensor and $D\in \mathbb{R}^{n\times n}$ be a nonnegative diagonal matrix. Then $D \mathcal{M}$ is  semipositive tensor.
\end{corol}

\begin{remk}
The converse of the above corollary is not true in general.
\end{remk}

\begin{examp}
Let $\mathcal{M}\in \mathbb{T}_{4,2}$ be such that $m_{1111}= 1,\; m_{1121}= -1,\; m_{1112}= -3,\; m_{1222}= -1,\; m_{1211}=1,\; m_{2121}= -3,\; m_{2222}= 2$ and all other entries of $\mathcal{M}$ are zero. Then for $u =\left(\begin{array}{c} u_1 \\ u_2 \end{array} \right) \in \mathbb{R}^2$ we have $\mathcal{M}u^3 = \left(\begin{array}{c} u_1^3 -3u_1^2 u_2 -u_2^3 \\ -3u_1^2 u_2 + 2u_2^3 \end{array} \right).$ Clearly for $u =\left(\begin{array}{c} 1 \\ 1 \end{array} \right)$ we obtain $\mathcal{M}u^3 = \left(\begin{array}{c} -3 \\ -1 \end{array} \right).$
% Then $u_1(\mathcal{M}u^3)_1 =-3 <0$ and $u_2(\mathcal{M}u^3)_2 =-1<0.$
This implies $\mathcal{M}$ is not a semipositive tensor.

\NI Now we consider two nonnegative diagonal matrices $D_1=\left(\begin{array}{cc}
    0 & 0 \\
    0 & 3
\end{array} \right)$ and $D_2=\left(\begin{array}{cc}
    2 & 0 \\
    0 & 0
\end{array} \right).$ Let $D_1\mathcal{M} =\mathcal{B}$ and $D_2\mathcal{M} =\mathcal{C}.$ Then $\mathcal{B}=(b_{ijkl})\in \mathbb{T}_{4,2}$ with  $b_{2121}=-9,\; b_{2222}=6$ and all other entries of $\mathcal{B}$ are zero. For $u =\left(\begin{array}{c} u_1 \\ u_2 \end{array} \right) \in \mathbb{R}^2$ we have $\mathcal{B}u^3 =\left(\begin{array}{c} 0 \\ -9 u_1^2 u_2 + 6 u_2^3 \end{array} \right).$ Clearly $\mathcal{B}$ is a semipositive tensor.

\NI Also $\mathcal{C}=(c_{ijkl})\in \mathbb{T}_{4,2}$ with $c_{1111}= 2,\; c_{1121}= -2,\; c_{1112}= -6,\; c_{1222}= -2,\; c_{1211}=2$ and all other entries of $\mathcal{C}$ are zero. Then for $u =\left(\begin{array}{c} u_1 \\ u_2 \end{array} \right) \in \mathbb{R}^2$ we have $\mathcal{C}u^3 =\left(\begin{array}{c} 2u_1^3 -6u_1^2 u_2 -2u_2^3 \\ 0 \end{array} \right).$ Clearly $\mathcal{C}$ is a semipositive tensor.
\end{examp}

\begin{theorem}
Let $P\in \mathbb{R}^{n\times n}$ be a permutation matrix also let $\mathcal{M} \in \mathbb{T}_{r,n}.$ Then the tensor $\mathcal{M}$ is (strictly) semipositive tensor if and only if $P \mathcal{M} P^T$ is (strictly) semipositive tensor. 
\end{theorem}
\begin{proof}
Let $\mathcal{M} \in \mathbb{T}_{r,n}$ be a (strictly) semipositive tensor. Let
$P\in \mathbb{R}^{n\times n}$ be a permutation matrix such that for any vector $v \in \mathbb{R}^n,$ $(P v)_k = v_{\sigma(k)}$ where $\sigma$ is a permutation on the set of indices $I_n.$ Let $0\neq u \geq 0$ where $u \in \mathbb{R}^n.$ Let $v= P^T u.$ Then $v\geq 0.$ As $\mathcal{M}$ is a (strictly) semipositive tensor $\exists \; l$ such that $v_l >0$ and $(\mathcal{M}v^{r-1})_l\geq 0 \; (>0).$ Thus $(P \mathcal{M} P^T u^{r-1})_{\sigma(l)}=(\mathcal{M} P^T u^{r-1})_l = (\mathcal{M} v^{r-1})_l\geq 0 \; (>0).$ Also $u_{\sigma(l)}=(P^T u)_l=v_l >0.$ Therefore $P \mathcal{M} P^T$ is (strictly) semipositive tensor.

Conversely, suppose $P\mathcal{M}P^T$ is (strictly) semipositive tensor. We have $\mathcal{M} =P(P\mathcal{M}P^T)P^T$, since $P$ is a permutation matrix. Hence by the forward part of the proof $\mathcal{M}$ is (strictly) semipositive tensor.
\end{proof}

Now we prove the necessary and sufficient conditions for a tensor to be semipositive tensor.

\begin{theorem}
Let $\mathcal{M}, \mathcal{N} \in \mathbb{T}_{r,n}$ where $\mathcal{M}$ is a (strictly) semipositive tensor and $\mathcal{N}\geq \mathcal{O}$. Then $\mathcal{M} + \mathcal{N}$ is (strictly) semipositive tensor.
\end{theorem}
\begin{proof}
Suppose $\mathcal{M} \in \mathbb{T}_{r,n}$ is (strictly) semipositive tensor. Then for $0\neq u\geq 0 \; \exists \;l \in I_n$ such that $u_l>0$ and $(\mathcal{M}u^{r-1})_l\geq 0 \; (>0).$ Let $\mathcal{N} \in \mathbb{T}_{r,n}$ where $\mathcal{N}\geq \mathcal{O}.$ Then for $0\neq u \in \mathbb{R}^n_+$ we have $(\mathcal{N}u^{r-1})_i\geq 0,\; \forall \; i \in I_n.$ Therefore $((\mathcal{M} + \mathcal{N})u^{r-1})_l=(\mathcal{M}u^{r-1})_l + (\mathcal{N}u^{r-1})_l\geq0 \; (>0).$ Hence, $\mathcal{M} + \mathcal{N}$ is (strictly) semipositive tensor.
\end{proof}

Here we discuss a necessary and sufficient condition for (strictly) semipositive tensor.

\begin{theorem}
Let $\mathcal{M} \in \mathbb{T}_{r,n}$ be a tensor such that all its proper principal subtensors are (strictly) semipositive tensor. $\mathcal{M}$ is (strictly) semipositive tensor if and only if for all diagonal tensors $\mathcal{D}\in \mathbb{T}_{r,n}$ with $\mathcal{D} > \mathcal{O} \;(\geq \mathcal{O})$, the tensor $\mathcal{M} +\mathcal{D}$ does not have a positive null vector.
\end{theorem}
\begin{proof}

% we are proving the contrapositive statement.
Let there be a diagonal tensor $\mathcal{D}\in \mathbb{T}_{r,n}$ with $\mathcal{D} > \mathcal{O} \;(\geq \mathcal{O})$ such that $\mathcal{M} +\mathcal{D}$ has a positive null vector. i.e., There exists $0<u \in\mathbb{R}^n$ such that
\begin{align*}
  (\mathcal{M} +\mathcal{D})u^{r-1}=0
    & \implies \mathcal{M}u^{r-1} +\mathcal{D}u^{r-1}=0\\
    & \implies \mathcal{M}u^{r-1} = -\mathcal{D}u^{r-1} <0\; (\leq 0),
\end{align*}
which shows that $\mathcal{M}$ is not (strictly) semipositive tensor.% Then by contrapositive sense $\mathcal{M}$ is semipositive tensor.

Conversely, suppose $\mathcal{M}$ is not (strictly) semipositive tensor. We show that there exists a diagonal tensor $\mathcal{D}\in \mathbb{T}_{r,n}$ with $\mathcal{D} >\mathcal{O} \;(\geq \mathcal{O})$ such that the tensor $\mathcal{M} +\mathcal{D}$ has a positive null vector. Since $\mathcal{M}$ is not (strictly) semipositive tensor (but all proper principal subtensors are), $\exists\; u>0$ such that $\mathcal{M}u^{r-1} <0\; (\leq 0).$ Now we construct diagonal tensor $\mathcal{D}$ with diagonal entries being defined as $d_{i ...i} = -\frac{(\mathcal{M}u^{r-1})_i}{u_i^{r-1}}.$ Then $\mathcal{D} > \mathcal{O}\; (\geq \mathcal{O})$ and $\mathcal{M}u^{r-1} = -\mathcal{D}u^{r-1}.$ Therefore $\mathcal{M}u^{r-1} + \mathcal{D}u^{r-1} =(\mathcal{M} +\mathcal{D})u^{r-1} =0.$ Thus $\mathcal{M} + \mathcal{D}$ has a positive null vector, since $u>0.$
\end{proof}

Here we propose a necessary and sufficient condition for a semipositive tensor to be a strictly semipositive tensor.

\begin{theorem}
A tensor $\mathcal{M}\in \mathbb{T}_{r,n}$ is semipositive tensor if and only if for every $\delta >0,$ the tensor $\mathcal{M} + \delta \mathcal{I}$ is strictly semipositive tensor.
\end{theorem}
\begin{proof}
Let $\mathcal{M}\in \mathbb{T}_{r,n}$ be a semipositive tensor. The for every $0\neq u\geq 0, \; \exists$ an index $l$ such that $u_l >0$ and $(\mathcal{M}u^{r-1})_l \geq 0.$ Therefore $[(\mathcal{M}+ \delta \mathcal{I}) u^{r-1}]_l = (\mathcal{M}u^{r-1})_l + \delta u^{r-1}_l >0.$ Thus $\mathcal{M}+ \delta \mathcal{I}$ is a strictly semipositive tensor.

Conversely, suppose that $\mathcal{M}+ \delta \mathcal{I}$ is a strictly semipositive tensor for each $\delta >0.$ For arbitrarily chosen $0\neq u \geq 0$ consider a sequence $\{\delta_k \}$, where $\delta_k >0$ and $\delta_k$ converges to zero. Then for each $k,\; \exists\; l_k\in I_n$ such that $u_{l_k} >0$ and $[(\mathcal{M}+ \delta \mathcal{I}) u^{r-1}]_{l_k} >0.$ This implies $\exists \; l\in I_n$ such that $u_l >0$ and $(\mathcal{M}u^{r-1})_l + \delta_k u^{r-1}_l = [(\mathcal{M}+ \delta_k \mathcal{I}) u^{r-1}]_l >0$ for infinitely many $\delta_k.$ Since $\delta_k\rightarrow 0$ when $k\rightarrow \infty$ we conclude that $\mathcal{M}u^{r-1}_l \geq 0.$ Hence $\mathcal{M}$ is semipositive tensor.
\end{proof}

Let $I$ denote the identity matrix of order $n$ and $[0, I]$ denote all $n \times n$ diagonal matrices whose diagonal entries are in $[0, 1].$ Now we establish a necessary and sufficient condition for strictly semipositive tensor.

\begin{theorem}
Let $\mathcal{M}\in \mathbb{T}_{r,n}$ and $r$ is even. $\mathcal{M}$ is strictly semipositive tensor if and only if $\forall\; G \in [0, I]$ the tensor $G\mathcal{I}+(I−G)\mathcal{M}$ has no nonzero nonnegative null vector.
\end{theorem}
\begin{proof}
Let $\mathcal{M}$ be a strictly semipositive tensor and $\forall\; G \in [0, I], \; \exists\; 0\neq u\geq 0$ which is a null vector of the tensor $G\mathcal{I}+(I−G)\mathcal{M}.$ i.e. $(G\mathcal{I}+(I−G)\mathcal{M})u^{r-1} =0.$ Let $v= \mathcal{M}u^{r-1}.$ As $\mathcal{M}$ is strictly semipositive tensor $\exists\; l\in I_n$ such that  $u_l>0$ and $v_l=(\mathcal{M}u^{r-1})_l >0.$ Now 
\begin{align*}
    [G\mathcal{I} + (I-G)\mathcal{M}]u^{r-1} =0 &\implies g_{ll} u_l^{r-1} + (1-g_{ll})(\mathcal{M}u^{r-1})_l =0\\
                                     &\implies g_{ll} (u_l^{r-1} - v_l)= -v_l. 
\end{align*}
Note that if $u_l^{r-1} = v_l$ then $v_l = 0$ which contradicts the fact that $v_l>0.$ Therefore $u_l^{r-1} - v_l\neq 0$ and $g_{ll}= \frac{- v_l}{u_l^{r-1} - v_l}.$ Now if $u_l^{r-1} - v_l >0$ then we have $-v_l\geq0,$ since $g_{ll} \geq 0.$ Therefore $v_l\leq 0,$ which is again a contradiction, since $v_l>0$. 
Let $u_l^{r-1} - v_l <0.$ Now since $g_{ll}\leq 1,$ we have $\frac{- v_l}{u_l^{r-1} - v_l} \leq 1 \implies u_l^{r-1} - v_l \leq -v_l \implies u_l\leq 0$ (since $r$ is even) which contradicts the fact that $u_l>0.$ Hence no nonzero nonnegative null vector exists.

Conversely, let $\mathcal{M}$ be not a strictly semipositive tensor. Then $\exists\;0 \neq u \geq0$ such that $u_i (\mathcal{M}u^{r-1})_i \leq 0 \; \forall \;i \in I_n.$ Let $v=\mathcal{M}u^{r-1}.$ Then for $u_l>0$ we have either $v_l=0$ or $v_l<0.$ Note that if $u_l =0,$ then there are three choices for $v_l$, they are $v_l=0,\;v_l>0 \text{ or } v_l<0 .$ Now we define diagonal a matrix $G$ with the diagonal elements as
  $g_{ii} = \left\{\begin{array}{cc}
      0 & \text{if } u_i > 0,\; v_i = 0\\
   \frac{-v_i}{u_i^{r-1} -v_i} & \text{if } u_i > 0,\; v_i < 0\\
      0 & \text{if } u_i=0,\; v_i=0\\
      1 & \text{if } u_i=0,\; v_i \lessgtr 0
  \end{array} \right. .$ 
Note that if $u_i>0$ and $v_i<0,$ then $0\leq \frac{-v_i}{u_i^{r-1}-v_i}\leq 1.$ If not, then there are two possibilities either $\frac{-v_i}{u_i^{r-1} -v_i}<0$ or $\frac{-v_i}{u_i^{r-1}-v_i}> 1.$ Now $\frac{-v_i}{u_i^{r-1}-v_i}<0$ implies $v_i>0$, since $u_i^{r-1}-v_i>0.$ This leads to a contradiction.
% Actually $v_i>0$ but for $v_i =0$ it also gives an contradiction.
Again if $\frac{-v_i}{u_i^{r-1} -v_i}>1$ then we have $-v_i > u_i^{r-1}-v_i \implies u_i^{r-1}<0,$ which is again a contradiction. Thus we have $0\leq g_{ii}\leq 1$ and $[(I-G)\mathcal{M}u^{r-1}]_i = (1-g_{ii})v_i = -g_{ii} u_i^{r-1}= -(G \mathcal{I}u^{r-1})_i.$ Therefore $[G\mathcal{I} + (I-G)\mathcal{M}]u^{r-1}=0.$ i.e. There exists a diagonal matrix $G$ with diagonal entries from $[0,1]$ such that the tensor $[G\mathcal{I} + (I-G)\mathcal{M}]$ has a nonzero nonnegative null vector.
\end{proof}

Here we establish a relation between an even order row diagonal (strictly) semipositive tensor and its majorization matrix.

\begin{theorem}
Let $\mathcal{M}\in \mathbb{T}_{r,n}$ be a row diagonal tensor of even order. $\mathcal{M}$ is (strictly) semipositive tensor if and only if $\Tilde{\mathcal{M}}\in \mathbb{R}^{n\times n}$ is a (strictly) semimonotone matrix.
\end{theorem}
\begin{proof}
Let $\mathcal{M}\in \mathbb{T}_{r,n}$ be an even order row diagonal tensor. i.e., $r$ is even. Let $\Tilde{\mathcal{M}}\in \mathbb{R}^{n\times n}$ be the majorization matrix of the tensor $\mathcal{M}.$ Then for $u \in \mathbb{R}^n$ we have $\mathcal{M} u^{r-1} = \Tilde{\mathcal{M}}\mathcal{I} u^{r-1}= \Tilde{\mathcal{M}}u^{[r-1]}.$ Let $\mathcal{M}$ be a (strictly) semipositive tensor. We prove that $\Tilde{\mathcal{M}}$ is (strictly) semimonotone matrix. Now for $0\neq v \geq 0,$ $v^{[\frac{1}{r-1}]}$ exists uniquely with $0\neq v^{[\frac{1}{r-1}]}\geq 0,$ since $r$ is even. Let $u=v^{[\frac{1}{r-1}]}.$ Since $\mathcal{M}$ is a (strictly) semipositive tensor, for $0\neq u \geq 0\; \exists \;l$ such that $u_l >0$ and $(\mathcal{M} u^{r-1})_l =(\Tilde{\mathcal{M}}u^{[r-1]})_l \geq 0\; (>0 )$ This implies $u_l^{r-1} >0$ and $(\Tilde{\mathcal{M}}u^{[r-1]})_l \geq 0 \;(>0).$ i.e., For $0\neq v \geq 0\; \exists \;l$ such that $ v_l=u_l^{r-1} >0$ and $(\Tilde{\mathcal{M}}v)_l =(\Tilde{\mathcal{M}}u^{[r-1]})_l \geq 0 \;(>0).$ This implies that $\Tilde{\mathcal{M}}$ is a (strictly) semimonotone matrix.

Conversely, let $\Tilde{\mathcal{M}}$ be a (strictly) semimonotone matrix. Since $\mathcal{M}= \Tilde{\mathcal{M}}\mathcal{I},$ for $x\in \mathbb{R}^n$ we have
\begin{equation}\label{equation for row diag}
    (\mathcal{M}x^{r-1})_i =(\Tilde{\mathcal{M}}x^{[r-1]})i, \; \forall \; i\in I_n.
\end{equation}
We choose $x \in \mathbb{R}^n$ such that $0\neq x\geq 0$ and construct the vector $y=x^{[r-1]}.$ Then $0\neq y \geq 0.$ Since $\Tilde{\mathcal{M}}$ is a (strictly) semimonotone matrix $\exists \; t \in I_n$ such that $y_t >0$ and $(\Tilde{\mathcal{M}}y)_t \geq 0 \;(>0).$ i.e., For $0\neq x\geq 0\; \exists\; t \in I_n$ such that $x^{r-1}_t >0$ and $(\mathcal{M}x^{r-1})_t \geq 0 \;(>0)$ by equation \ref{equation for row diag}. Therefore for $0\neq x\geq 0\; \exists\; t \in I_n$ such that $x_t >0$ (since $r$ is even) and $(\mathcal{M}x^{r-1})_t \geq 0 \;(>0).$ Hence $\mathcal{M}$ is (strictly) semipositive tensor.
\end{proof}

Not all even order semipositive tensors are row diagonal. Here we give an example of a semipositive tensor which is not row diagonal tensor.
\begin{examp}
Let $\mathcal{M}\in \mathbb{T}_{4,3}$ be such that $m_{1122}= 2,\; m_{1131}= 2,\; m_{2211}=1,\; m_{2112}= -4,\; m_{2322}=-1,\; m_{3232}= -1,\; m_{3322}=1,\; m_{3313}=3$ and all other entries of $\mathcal{M}$ are zero. Then for $u =\left(\begin{array}{c} u_1 \\ u_2\\ u_3 \end{array} \right) \in \mathbb{R}^3$ we have $\mathcal{M}u^3 = \left(\begin{array}{c} 2u_1 u_2^2 + 2u_1^2 u_3 \\ -3u_1^2 u_2 - u_2^2 u_3 \\ 3u_1 u_3^2 \end{array} \right).$
The tensor $\mathcal{M}$ is a semipositive tensor of even order but $\mathcal{M}$ is not a row diagonal tensor.
\end{examp}

\section{Conclusion}
In this article, we establish some tensor theoretic properties of semipositive tensor and strictly semipositive tensor. We show invariance property of (strictly) semipositive tensor with respect to diagonal matrix and permutation matrix. We show that $\mathcal{M}\in \mathbb{T}_{r,n}$ is semipositive tensor if and only if for every $\delta >0,$ the tensor $\mathcal{M} + \delta \mathcal{I}$ is strictly semipositive tensor. We propose a necessary and sufficient conditions for (strictly) semipositive tensor. Furthermore, we prove that an even order row diagonal tensor is (strictly) semipositive tensor if and only if its majorization matrix is (strictly) semimonotone matrix.

\section{Acknowledgment}
%The author A. Dutta is thankful to the Department of Science and technology, Govt. of India, INSPIRE Fellowship Scheme for financial support.
The author R. Deb is thankful to the Council of Scientific $\&$ Industrial Research (CSIR), India, Junior Research Fellowship scheme for financial support.

\bibliographystyle{plain}
\bibliography{referencesall}

\end{document}